\documentclass[a4paper,12pt,reqno]{amsart}

\usepackage[
  margin=30mm,
  marginparwidth=25mm,     
  marginparsep=2mm,       
  bottom=25mm,
  ]{geometry}

\usepackage[T1]{fontenc}
\usepackage[utf8]{inputenc}
\usepackage[english]{babel}
\usepackage{amsmath,amssymb,amsthm,mathtools}
\usepackage{latexsym}
\usepackage{delarray}
\usepackage{bbm}
\usepackage{scrextend}
\usepackage{hyperref}
\usepackage{datetime}
\usepackage{mathrsfs}
\usepackage{enumitem}  
\usepackage{tikz}
\usepackage{enumerate}
\usepackage{comment}

\newtheorem{Th}{Theorem}[section]
\newtheorem{Prop}[Th]{Proposition}
\theoremstyle{definition}
\newtheorem{Lem}[Th]{Lemma}

\newtheorem{Rem}[Th]{Remark}
\theoremstyle{definition}
\newtheorem{Def}[Th]{Definition}

\numberwithin{equation}{section} 
 \normalfont

\renewcommand{\leq}{\leqslant}
\renewcommand{\geq}{\geqslant}

\allowdisplaybreaks

\begin{document}
\title[Higher order parabolic equation]
   {A note on lifespan estimates for higher-order parabolic equations}

\author[N.N. Tobakhanov]{Nurdaulet N. Tobakhanov}
\address{
  Nurdaulet N. Tobakhanov:
  \endgraf
  Department of Mathematics
  \endgraf
Nazarbayev University, Astana, Kazakhstan
  \endgraf
  {\it } {\rm nurdaulet.tobakhanov@nu.edu.kz}
  }

 \author[B.T. Torebek]{Berikbol T. Torebek}
\address{
	Berikbol T. Torebek:
	\endgraf
Institute of
Mathematics and Mathematical Modeling
\endgraf
Shevchenko street 28, 050010 Almaty, Kazakhstan
	\endgraf {\it} {\rm torebek@math.kz}
} 


\thanks{Corresponding author: Berikbol T. Torebek, email: torebek@math.kz}



     \begin{abstract} 
     We investigate the lifespan of solutions to the higher-order semilinear parabolic equation 
     $$ u_t+(-\Delta)^m u=|u|^p, \quad x \in \mathbb{R}^n, \,t>0$$
     with initial data $\varepsilon u_0$. We focus on the precise asymptotic behavior of the lifespan of nontrivial solutions. By combining the test function method and semigroup estimates, we derive both upper and lower bounds for the lifespan of solutions\begin{equation*}T_{\varepsilon} \simeq \begin{cases}\varepsilon^{-\left(\frac{1}{p-1}-\frac{n}{2m}\right)^{-1}}, & 1<p<p_{\text {Fuj }}, \\ \exp \left( \varepsilon^{-(p-1)}\right), & p=p_{\text {Fuj }},\\ \infty, &p>p_{\text {Fuj }},\end{cases}\end{equation*} where $p_{Fuj}=1+\frac{2m}{n}$ is the critical exponent of Fujita. These estimates refine and extend the earlier results of Caristi–Mitidieri [J. Math. Anal. Appl., 279:2 (2003), 710--722] and Sun [Electron. J. Differential Equations, 17 (2010)], who obtained only upper bounds for slowly decaying initial data. In our setting, the above condition on the initial data is replaced by the assumption $u_0\in L^1\cap L^\infty$, which sharpens the results of the aforementioned works.
      \end{abstract}

     \maketitle
\tableofcontents

\section{Introduction and main result}
In this paper, we study
the higher-order semilinear parabolic equation
\begin{equation}\label{main}
  \begin{cases}u_t+(-\Delta)^{m} u=|u|^p, & x \in \mathbb{R}^n, t>0, \\{}\\ u(x, 0)=\varepsilon u_0(x), & x \in \mathbb{R}^n,\end{cases}  
\end{equation}
where $m\in \mathbb{N}$ , $\varepsilon>0$ is a small parameter, and $p>1$. 

This model naturally arises in various physical and engineering contexts as a higher-order diffusion equation. For instance, it describes thin film dynamics under surface tension, see Bernis and Friedman \cite{apl} and the monograph \cite{mon}.

The study of semilinear parabolic equations of the form
\begin{equation}\label{poly}
u_t-\Delta u=|u|^p
\end{equation}
has attracted much attention in recent decades. For the classical semilinear heat equation, Fujita \cite{Fujita} identified the critical exponent  $p_F=1+\frac{2}{n}$ which separates the global existence and blow-up regimes for nontrivial, nonnegative solutions (see also \cite{Hayakawa, SG, koba}). Subsequently, Lee and Ni \cite{t2} obtained precise estimates for the lifespan of solutions, where the lifespan denotes the maximal time interval of the solution. In particular, they proved that if the initial data decay sufficiently fast, then the lifespan of the local solution satisfies the sharp asymptotic estimate
$$
T_{\varepsilon} \simeq \begin{cases} \varepsilon^{-\left(\frac{1}{p-1}-\frac{n}{2}\right)^{-1}}, & 1<p<p_F, \\{}\\ \exp \left( \varepsilon^{-(p-1)}\right), &  p=p_F .\end{cases}
$$
A detailed analysis of lifespan estimates for problem \eqref{poly}, under various classes of initial data, is presented in \cite{Tayachi} and the references therein.

Compared with the classical heat equation, the higher-order case presents analytical difficulties due to the lack of the maximum principle and the sign-changing nature of its fundamental solution. However,  Egorov et al. \cite{Egor} showed that if $$1< p \leq p_{\text{Fuj}}:=1+\frac{2m}{n}$$ and the initial data satisfy $\int_{\mathbb{R}^n} u_0(x) d x \geq 0$, then any nontrivial solution to problem \eqref{main} blows up in finite time. Moreover, it was proved by Galaktionov and Pohozaev \cite{Gala} that for $p>p_{\text {Fuj }}$, the corresponding problem admits global solutions provided the initial data are sufficiently small. In \cite{Zaag}, the authors showed that the equation admits type-I blow-up solutions and a characterization of the corresponding asymptotic behavior. Several studies have addressed blow-up solutions in the setting of polyharmonic operators (see, for example \cite{Alsaedi, Fino2, SunShi}). For a detailed analysis of the Cauchy problems for polyharmonic heat equations with singular initial data, the reader is referred to \cite{Ishige,Miyake}.

It is worth mentioning that Caristi and Mitidieri~\cite{Caristi}, and later Sun~\cite{SunFuqin}, investigated the higher-order semilinear parabolic equation~\eqref{main} under slowly decaying initial data.
They proved that any nontrivial solution blows up in finite time and established an upper bound for the lifespan 
$$
T_{\varepsilon} \lesssim  \varepsilon^{-\left(\frac{1}{p-1}-\frac{\kappa}{2 m}\right)^{-1}}.
$$
The analysis was carried out for initial data satisfying
\begin{equation}\label{U}
u_0(x) \geq\left\{\begin{array}{ll}
\delta, & |x| \leq \varepsilon_0, \\{}\\
C_0|x|^{-\kappa}, & |x|>\varepsilon_0,
\end{array} \right.\end{equation}
where $\delta>0,$ $\varepsilon_0>0,$ and the decay rate parameter $\kappa$ lies in the range
$$
\kappa= \begin{cases}n<\kappa<2 m /(p-1), & 1<p<p_{\text {Fuj }}, \\{}\\ 0<\kappa<n, & p=p_{\text {Fuj }}.\end{cases}
$$
These results reveal that the spatial decay of the initial data affects the blow-up time, indicating that slower decay leads to faster blow-up.

In this work, we refine and extend the above results by considering initial data $u_0 \in L^1\left(\mathbb{R}^n\right) \cap L^{\infty}\left(\mathbb{R}^n\right)$ without any decay assumption.
We establish both upper and lower estimates for the lifespan of solutions to problem \eqref{main}, including the critical Fujita case $p=p_{\text {Fuj }}$.
Our analysis is based on the test function method and the semigroup theory.
These findings offer a more detailed understanding of the blow-up mechanism and extend several known results for the higher-order semilinear parabolic equations.

To proceed, we first recall the definition of a solution to problem \eqref{main}.
\begin{Def} Given $u_0 \in L^1\left(\mathbb{R}^n\right) \cap L^{\infty}\left(\mathbb{R}^n\right)$, a function $$u \in C\left(\left[0, T_{\varepsilon}\right) ; L^1\left(\mathbb{R}^n\right) \cap L^{\infty}\left(\mathbb{R}^n\right)\right)$$ is said to be a mild solution of \eqref{main} if it satisfies the integral equation
\begin{equation}\label{mild}
   u(x, t)= \varepsilon e^{-t(-\Delta)^m}u_0+\int_0^t  e^{-(t-s)(-\Delta)^m}|u(s)|^p ds,
\end{equation}  where $e^{-t(-\Delta)^m}$ denotes the semigroup associated with the polyharmonic heat equation; see Section \ref{section2} for further details.

Moreover, the lifespan of a mild solution is given by 
    $$
\begin{aligned}
 T_{\varepsilon}:=\sup \{T>0\mid \exists \quad &\text {a unique solution }\\
 &\left.\quad u \in C\left([0, T) ; L^1\left(\mathbb{R}^n\right) \cap L^{\infty}\left(\mathbb{R}^n\right)\right) \text { to }\eqref{main}\right\} .
\end{aligned}
$$
\end{Def}
\begin{Th}\label{lf}
     If $u_0 \in L^1\left(\mathbb{R}^n\right) \cap L^{\infty}\left(\mathbb{R}^n\right)$ and $\int_{\mathbb{R}^n} u_0(x) d x > 0$, then for sufficiently small
\(\varepsilon>0\), problem \eqref{main} admits a unique mild solution
$$
u \in C\left(\left[0, T_{\varepsilon}\right) ; L^1\left(\mathbb{R}^n\right)\right) \cap L^{\infty}\left(\left(0, T_{\varepsilon}\right) ; L^{\infty}\left(\mathbb{R}^n\right)\right)
$$
for some $T_{\varepsilon}>0$. Moreover,  the lifespan satisfies
\begin{equation}\label{LS}
T_{\varepsilon} \simeq \begin{cases}\varepsilon^{-\left(\frac{1}{p-1}-\frac{n}{2m}\right)^{-1}}, & 1<p<p_{\text {Fuj }}, \\{}\\ \exp \left( \varepsilon^{-(p-1)}\right), & p=p_{\text {Fuj }},\,\\{}\\\infty, &p>p_{\text{Fuj}}.\end{cases}\end{equation}
\end{Th}
\begin{Rem} The results of Theorem \ref{lf} strengthen and extend those of Caristi et al. \cite{Caristi} and Sun \cite{SunFuqin} in the following ways:
\begin{itemize}
    \item[(i)] The results of Caristi and Mitidieri \cite{Caristi} and Sun \cite{SunFuqin} provided only an upper bound for the lifespan of solutions to problem~\eqref{main} under slowly decaying nonnegative initial data $u_0(x)$ satisfying \eqref{U}. In our setting, the above condition on the initial data is replaced by the assumption $u_0\in L^1(\mathbb{R}^n)\cap L^\infty(\mathbb{R}^n)$, which sharpens the results of the aforementioned works.
    \item[(ii)] Caristi and Mitidieri \cite{Caristi}, and Sun \cite{SunFuqin} established the upper bound for the lifespan 
$$T_{\varepsilon} \lesssim  \varepsilon^{-\left(\frac{1}{p-1}-\frac{\kappa}{2 m}\right)^{-1}},\,\kappa= \begin{cases}n<\kappa<2 m /(p-1), & 1<p<p_{\text {Fuj}}, \\{}\\ 0<\kappa<n, & p=p_{\text {Fuj}}.\end{cases}$$ In our setting, the above estimate becomes sharper in the subcritical regime $p<p_{Fuj},$ and in the critical case $p=p_{Fuj},$ it further improves to the bound $T_\varepsilon\simeq \exp\left(-\varepsilon^{-(p-1)}\right).$
\end{itemize}
\end{Rem}
\begin{Rem}
When $m=1$, estimate \eqref{LS} coincides with the lifespan estimates for the classical semilinear heat equation established by Lee and Ni in \cite{t2}. Consequently, our results include the Lee–Ni estimates as a particular instance. We also note that the authors of this paper have recently obtained in \cite{Torebek} an upper bound for the lifespan estimate \begin{equation*}
T_{\varepsilon} \lesssim \begin{cases}\varepsilon^{-\left(\frac{1}{p-1}-\frac{n}{4}\right)^{-1}}, & 1<p<1+\frac{4}{n}, \\{}\\ \exp \left(\varepsilon^{-(p-1)}\right), & p=1+\frac{4}{n},\end{cases}\end{equation*} for the semilinear biharmonic heat equation in an exterior domain $\Omega^c$, which is a particular case of \eqref{LS} when $\Omega^c\equiv\mathbb{R}^n$. To the best of the authors’ knowledge, estimate \eqref{LS} is new even for $m=2$, i.e., for the semilinear biharmonic heat equation $u_t+\Delta^2u=|u|^p.$
\end{Rem}

The proof of the upper bound for the lifespan relies on the test-function method developed by Ikeda and Sobajima \cite{Ikeda1}, together with an application of Faà di Bruno’s formula. In contrast, the lower-bound proof is obtained via a fixed-point argument that makes essential use of the properties of the heat kernel associated with the polyharmonic heat equation.

\section{Preliminaries}\label{section2}
We consider the Cauchy problem for the linear polyharmonic heat equation on
$\mathbb{R}^n$
\begin{equation*}
    \partial_t u + (-\Delta)^m u = 0,\,\,t>0, x\in\mathbb{R}^n,
    \qquad
    u(0,x)=u_0(x),\,\,x\in\mathbb{R}^n,
\end{equation*}
where $m\in\mathbb{N}$. The associated semigroup is denoted by
\[
    S_m(t)u_0 := e^{-t(-\Delta)^m}u_0.
\]
Taking the Fourier transform in space gives
\[
    \widehat{S_m(t)u_0}(\xi)
    =
    e^{-t|\xi|^{2m}}\widehat{u_0}(\xi),\,\xi\in\mathbb{R}^n.
\]
Thus $S_m(t)$ is a convolution operator,
\[
    S_m(t)u_0(x)
    =
    K_m(t,\cdot)*u_0(x),
\]
where the polyharmonic heat kernel is given by
\begin{equation}\label{m-heat-kernel}
    K_m(t,x)
    =
    \frac{1}{(2\pi)^n}
    \int_{\mathbb{R}^n}
    e^{ix\cdot \xi}
    e^{-t|\xi|^{2m}}
    \,d\xi .
\end{equation}
Since the integrand is radial in \(\xi\), we may write \(r=|x|\) and use the identity
\[
\int_{S^{n-1}} e^{i r\rho \cos\theta}\,d\omega
=
(2\pi)^{n/2}
(r\rho)^{1-n/2}
J_{n/2-1}(r\rho),
\]
where \(J_\nu\) is the Bessel function of the first kind.

\noindent Hence, we have
\[
{
K_m(t,x)
=
\frac{1}{(2\pi)^{n/2}}
|x|^{1-n/2}
\int_0^\infty
e^{-t\rho^{2m}}
\rho^{n/2}
J_{n/2-1}(|x|\rho)
\,d\rho.
}
\]

Below we present the $L^q-L^r$-boundedness of the polyharmonic heat semigroup.
\begin{Prop}\label{estimate}
Let \(f\in L^q(\mathbb R^n)\). Then there exists a constant \(C>0\) such that
$$
\left\|e^{-t(-\Delta)^m} f\right\|_{L^r(\mathbb{R}^n)} \leq C t^{-\frac{n}{2 m}\left(\frac{1}{q}-\frac{1}{r}\right)}\|f\|_{L^q(\mathbb{R}^n)}
$$
for all $1 \leq q \leq r \leq \infty$, and $t>0.$
\end{Prop}
It is very likely that the preceding statement is already established. For instance, a comparable result for $m=2$ was presented in \cite[Prop. 6.1]{slim}. Nevertheless, for the reader’s convenience, we include the proof of Proposition \ref{estimate} below.
\begin{proof}[Proof of Proposition \ref{estimate}]
The polyharmonic heat kernel satisfies the scaling relation
\begin{equation}
    K_m(t,x)
    =
    t^{-\frac{n}{2m}}
    K_m\left(1,t^{-\frac{1}{2m}}x\right).
\end{equation}
Consequently, for every $1\le k\le \infty$, there exists a constant
$C_k>0$ such that
\begin{equation}
    \|K_m(t,\cdot)\|_{L^k(\mathbb{R}^n)}
    \le
    C_k
    t^{-\frac{n}{2m}\left(1-\frac1k\right)}.
\end{equation}

We now prove the standard $L^q$--$L^r$ estimate. Let
\[
    1\le q\le r\le \infty.
\]
Choose $k\in[1,\infty]$ such that
\begin{equation}
    1+\frac1r=\frac1q+\frac1k.
\end{equation}
By Young's convolution inequality, we have
\begin{align*}\|S_m(t)f\|_{L^r(\mathbb{R}^n)}&
    =
    \|K_m(t,\cdot)*f\|_{L^r(\mathbb{R}^n)}
   \\& \le
    \|K_m(t,\cdot)\|_{L^k(\mathbb{R}^n)}
    \|f\|_{L^q(\mathbb{R}^n)}.\end{align*}
Since
\[
    1-\frac1k=\frac1q-\frac1r,
\]
we obtain
\begin{equation}
    \|S_m(t)f\|_{L^r(\mathbb{R}^n)}
    \le
    C
    t^{-\frac{n}{2m}\left(\frac1q-\frac1r\right)}
    \|f\|_{L^q(\mathbb{R}^n)}
\end{equation}
for all $t>0$, where $C>0$ may depend only on $n,m,q,r$. The proof is complete.
\end{proof}

Let us now define the cut-off functions 
\begin{equation}\label{psi}
  \psi_R(x, t):=\left[\phi\left(\frac{|x|^{2m}+t}{R}\right)\right]^{2mp^{\prime}}, \quad \psi_R^*(x, t):=\left[\phi^*\left(\frac{|x|^{2m}+t}{R}\right)\right]^{2mp^{\prime}}, 
\end{equation}

\noindent for all $(x, t) \in Q:=(0, \infty) \times \mathbb{R}^n$, where  $\phi \in C^{\infty}([0, \infty)) $ satisfies 
$$
\phi(s)=\left\{\begin{array}{ll}
1 & \text { if } s \in[0,1 / 2], \\
\searrow & \text { if } s \in(1 / 2,1), \\
0 & \text { if } s \in[1, \infty),
\end{array}\right. \quad \phi^*(s)= \begin{cases}0 & \text { if } s \in[0,1 / 2), \\
\phi(s) & \text { if } s \in[1 / 2, \infty).\end{cases}$$

We now state a result concerning certain properties of the cutoff functions.
\begin{Lem}\label{lifespanes}
 Let $\psi_R$ and $\psi_R^*$ be the cutoff functions defined in \eqref{psi}. For all sufficiently large $R>0$, there exists a constant $C>0$ such that
$$
\left|\partial_t \psi_R(t, x)\right|+\left|\Delta^m \psi_R(t, x)\right| \leq C R^{-1}\left[\psi_R^*(t, x)\right]^{1 / p}, \quad(t, x) \in Q.
$$
\end{Lem}
\begin{proof}
    Let us introduce  scaled space–time variables
    $$y=\frac{x}{R^{1 /(2 m)}}, \quad \tau=\frac{t}{R}, \quad s(y,\tau):=|y|^{2m}+\tau, \quad s_R(x,t)=\frac{|x|^{2 m}+t}{R}.$$
The differential operators scale as
$$
\begin{aligned}
\partial_t \psi_R(s_R)=R^{-1} \partial_{\tau}\psi(s), \quad \partial_{x_i}\psi_R(s_R)=R^{-1 /(2 m)} \partial_{y_i}\psi(s),
\end{aligned}
$$
and therefore
\begin{equation}\label{operator}
  \Delta_x^m \psi_R(s_R)=R^{-1} \Delta_{y}^m\psi(s).  
\end{equation}
Let us briefly compute the time derivative
    $$
\begin{aligned}
\left|\partial_t \psi_R(t, x)\right|=R^{-1} \left|\partial_{\tau}\psi(s)\right| &=2m p^{\prime} R^{-1} \left[\phi\left(s_R\right)\right]^{2mp^{\prime}-1} \phi^{\prime}\left(s_R\right)\\
&\leq CR^{-1} \left[\phi^*\left(s_R\right)\right]^{\frac{2m p^{\prime}}{p}} [\phi\left(s_R\right)]^{2m-1}\left|\phi^{\prime}\left(s_R\right)\right| \\
& \leq CR^{-1}\left[\psi_R^*(t, x)\right]^{\frac{1}{p}} .
\end{aligned}
$$
The estimate of the test function with the polyharmonic operator begins by viewing
\[
\psi(s)=(h\circ g)(s),\qquad g(s)=\phi(s),\qquad h(z)=z^{2mp'},
\]
as a composition. By the multivariable version of the Faà di Bruno formula (\cite{Ozawa}
Proposition A.1), for any multi-index \(\alpha\) with \(|\alpha|\ge1\) one has
$$
\partial_y^\alpha(\psi(s))=\alpha!\sum_{k=1}^{|\alpha|} \sum_{\substack{\gamma_1+\cdots+\gamma_k\le\alpha\\
|\gamma_1|+\cdots+|\gamma_k|=|\alpha|}} \frac{1}{k!}\left( h^{(k)}(g(s))\right) \prod_{i=1}^k \frac{\partial_y^{\gamma_i}(g(s))}{\gamma_i!}.
$$
By applying the chain rule to each factor, we obtain
\[
\partial_y^{\gamma_i}(g(s))=\sum_{r=1}^{|\gamma_i|} g^{(r)}(s)
\sum_{\substack{\beta_1+\cdots+\beta_r\le\gamma_i\\
|\beta_1|+\cdots+|\beta_r|=|\gamma_i|}} C_{\gamma_i,\beta}\prod_{j=1}^r\partial_y^{\beta_j}s,
\]
where \(s=s(y,\tau)=|y|^{2m}+\tau\). Since \(\phi\in C_c^\infty\), all derivatives \(g^{(r)}(s)=\phi^{(r)}(s)\) are bounded, and each derivative \(\partial_y^{\beta}s\) is a fixed polynomial in \(y,\tau\). Thus, every term in the Faà–di Bruno expansion is bounded by 
$$
\left|\partial_y^\alpha(\phi(s))^{2mp'}\right| \leq C_{\alpha,\gamma,\beta} \sum_{k=1}^{|\alpha|} 2mp'(2mp'-1) \ldots(2mp'-k+1)(\phi(s))^{2mp'-k}.
$$
Therefore, as  $2mp'-k \geq 2mp'-|\alpha|$ for all $1 \leq k \leq|\alpha|$, we get
\begin{equation}\label{tuindi}
    \left|\partial_y^\alpha(\phi(s))^{2mp'}\right| \leq C_{\alpha,m,p}(\phi(s))^{2mp'-|\alpha|}, \quad \text { for all } |y| \leq 1.
\end{equation}
Taking \(|\alpha|=2m\) in \eqref{tuindi} and using \eqref{operator}, we obtain
\[
\begin{aligned}
\big|\Delta_x^m\psi_R(s_R)\big|
&=R^{-1}\big|\Delta_y^m\psi(s)\big|\\
& \leq C_m  R^{-1} \sum_{|\alpha|=m}\left|\partial_y^{2 \alpha}(\phi(s))^{2mp'}\right|\\
&\le C_{\alpha,m,p} R^{-1}\,\phi(s)^{2mp'-2m}\\
&\le C_{\alpha,m,p} R^{-1}\,[\psi_R^*(s_R)]^{1/p},
\end{aligned}
\]
where the last inequality follows from the definitions of \(\psi_R=\phi^{2mp'}(s_R)\) and the auxiliary cutoff function \(\psi_R^*(s_R)\) defined in \eqref{psi}. This completes the proof.
\end{proof}

\section{Proof of Theorem \ref{lf}}
\textit{$\bullet$ Lower estimate.} To establish a lower bound of the lifespan, we employ the standard contraction mapping principle.
We first introduce the auxiliary function
$$
M(t):=\sup _{\tau \in[0, t)} \left\{(1+\tau)^{\frac{n}{2m}}\left\|u(\tau)\right\|_{L^{\infty}}+\left\|u(\tau)\right\|_{L^1}\right\}.
$$
Starting from the mild solution representation \eqref{mild} and using the $L^{1}-L^{1}$ estimate, we derive
$$
\begin{aligned}
  \left\|u(t)\right\|_{L^1} \leq C \varepsilon\left\|u_{0}\right\|_{L^1}+C \int_0^t\left\|u(\tau)\right\|_{L^{p}}^{p} d \tau. 
\end{aligned}
$$
By using  Hölder’s inequality
$$\|u(\tau)\|_{L^p}^p \leq\|u(\tau)\|_{L^1}\|u(\tau)\|_{L^\infty}^{p-1},$$
and the definition of $M(t)$, we get
$$
\begin{aligned}
  \left\|u(t)\right\|_{L^1}  \leq  C \varepsilon+C M(t)^{p}\int_0^t(1+\tau)^{-\frac{n(p-1)}{2m}}  d \tau.
\end{aligned}
$$
We now  derive an estimate for $\left\|u(t)\right\|_{L^{\infty}}$. When $0<t \leq 1$, the $L^{\infty}-L^{\infty}$ estimate applied to the mild formulation leads to
$$
\begin{aligned}
(1+t)^{\frac{n}{2m}}\left\|u(t)\right\|_{L^{\infty}} & \leq C \varepsilon \left\|u_{0}\right\|_{L^{\infty}}+C \int_0^t\left\|u(\tau)\right\|_{L^{\infty}}^{p} d \tau \\
& \leq C \varepsilon+C \int_0^1(1+\tau)^{-\frac{np}{2m}} M(\tau)^{p} d \tau \\
& \leq C \varepsilon +C M(t)^{p}\int_0^1(1+\tau)^{-\frac{np}{2m}}  d \tau\\
& = C \varepsilon +C M(t)^{p},
\end{aligned}
$$
where we used the definition of $M(t)$. When $t \geq 1$, the $L^1-L^{\infty}$ estimate gives

$$
\begin{aligned}
(1+t)^{\frac{n}{2m}}\left\|u(t)\right\|_{L^{\infty}} &\leq  C \varepsilon\left\|u_{0}\right\|_{L^1}  +C(1+t)^{\frac{n}{2m}} \int_0^{t / 2}(t-\tau)^{- \frac{n}{2m}}\left\|\left|u(\tau)\right|^{p}\right\|_{L^1} d \tau \\
& \qquad +C(1+t)^{\frac{n}{2m}} \int_{t / 2}^t\left\|u(\tau)\right\|_{L^{\infty}}^{p} d \tau \\
&\leq C\varepsilon\left\|u_{0}\right\|_{L^1} +C \int_0^{t / 2}(1+\tau)^{-\frac{n(p-1)}{2m}} M(\tau)^{p} d \tau \\
&\qquad +C(1+t)^{\frac{n}{2m}} \int_{t / 2}^t(1+\tau)^{-\frac{np}{2m}} M(\tau)^{p} d \tau \\
& \leq C\varepsilon\left\|u_{0}\right\|_{L^1}+C  M(t)^{p}\int_0^t(1+\tau)^{-\frac{n(p-1)}{2m}}  d \tau.
\end{aligned}
$$
Collecting both the $L^1$  and $L^{\infty}$ estimates, we arrive at the following inequality:
\begin{equation}\label{M-key}
M(t) \leq C_0 \varepsilon +C_1 M(t)^{p}\int_0^t(1+\tau)^{-\frac{n(p-1)}{2m}}  d \tau
\end{equation}
for some constants $C_0, C_1>0$.
Let $T_1$ denote the smallest time such that $M(t)=2C_0\varepsilon$. Then, substituting $t=T_1$ in the above inequality, we have
$$
 \varepsilon \leq \begin{cases}C\varepsilon^{p} (1+T_1)^{(\frac{1}{p-1}-\frac{n}{2m})(p-1)}  & \text { if } 1<p<p_{\text {Fuj }}  \text {, } \\{}\\
C  \varepsilon^{p} \log(T_1+1) & \text { if } p=p_{\text {Fuj }},\end{cases}
$$
which, together with the smallness of $ \varepsilon$, implies 
$$
T_{\varepsilon}\geq T_1 \geq \begin{cases}C\varepsilon^{\left(-\frac{1}{p-1}-\frac{n}{2m}\right)^{-1}} & \text { if } 1<p<p_{\text {Fuj }}  \text {, }  \\{}\\
\operatorname{exp}\left(C\varepsilon^{-(p-1)}\right) & \text { if } p=p_{\text {Fuj }}.\end{cases}
$$

If \(p>p_{\text{Fuj}}\), then \(\frac{n(p-1)}{2m}>1\), so that
\[
\int_0^t (1+\tau)^{-\frac{n(p-1)}{2m}}\,d\tau < \infty,\,\, \text{for all}\,\,\,t>0.
\]
Consequently, \eqref{M-key} gives the uniform bound
\[
M(t)\le C_0\varepsilon + C_4 M(t)^p,\qquad t\ge 0.
\]
Thus, by a standard fixed-point argument, for sufficiently small
\(\varepsilon>0\), the associated mild solution is global; that is,
\[
T_\varepsilon=+\infty .
\]

\textit{$\bullet$ Upper estimate.} The upper bound estimate is derived by employing the approach proposed by Ikeda and Sobajima \cite{Ikeda1}, suitably adapted to the problem \eqref{main}.
\noindent We now focus on the representation of the mild solution.
Recalling equation \eqref{mild}, we multiply both sides by the test function $\psi_R$ and integrate over $\mathbb{R}^n$. This gives
\begin{equation*}\begin{split}
\int_{\mathbb{R}^n}u\psi_R dx&= \varepsilon\int_{\mathbb{R}^n}e^{-t(-\Delta)^m}u_0(x)\psi_R(x,0) dx\\&+ \int_{\mathbb{R}^n}\int_0^t e^{-(t-s)(-\Delta)^m}|u(s)|^p  ds\psi_R(x,t) dx.\end{split}
\end{equation*}
Differentiating both sides of the previous identity with respect to $t$, we obtain
\begin{equation*}
\begin{aligned}
    \frac{d}{dt}\int_{\mathbb{R}^n}u\psi_R dx&=\varepsilon\int_{\mathbb{R}^n}\frac{d}{dt}\left(e^{-t(-\Delta)^m}u_0\psi_R\right) dx\\&+\int_{\mathbb{R}^n}\frac{d}{dt}\left(\int_0^t e^{-(t-s)(-\Delta)^m}|u|^p  ds\psi_R \right)dx.
\end{aligned}
    \end{equation*}
We now compute each term separately. For the first one, we get
\begin{equation}\label{m1}
\begin{aligned}
    \varepsilon\int_{\mathbb{R}^n}\frac{d}{dt}&\left(e^{-t(-\Delta)^m}u_0\psi_R\right) dx
        \\& \qquad =-\varepsilon\int_{\mathbb{R}^n} \{(-\Delta)^m e^{-t (-\Delta)^m} u_0\} \psi_R dx+\varepsilon\int_{\mathbb{R}^n} e^{-t (-\Delta)^m} u_0 \partial_t\psi_R dx \\
& \qquad =-\varepsilon\int_{\mathbb{R}^n} e^{-t (-\Delta)^m} u_0(-\Delta)^m \psi_Rdx+\varepsilon\int_{\mathbb{R}^n} e^{-t (-\Delta)^m} u_0 \partial_t\psi_R dx,
\end{aligned}
    \end{equation}
where integration by parts has been applied in the first term. Similarly, for the second term, we have
\begin{equation}\label{m2}
\begin{aligned}
    \int_{\mathbb{R}^n}\frac{d}{dt}&\left(\int_0^t e^{-(t-s)(-\Delta)^m}|u|^p  ds\psi_R \right)dx\\
    &=\int_{\mathbb{R}^n} |u(t)|^p \psi_R dx-\int_{\mathbb{R}^n} \left(\int_0^t e^{-(t-s) (-\Delta)^m}\left(|u(s)|^p\right) d s\right)(-\Delta)^m\psi_R dx \\
&   +\int_{\mathbb{R}^n} \left(\int_0^t e^{-(t-s) (-\Delta)^m}\left(|u(s)|^p\right) d s\right)\partial_t \psi_R dx.
\end{aligned}
    \end{equation}
Combining \eqref{m1} and \eqref{m2}, we obtain
\begin{equation*}
\begin{aligned}
    \frac{d}{dt}\int_{\mathbb{R}^n}u\psi_R dx&-\int_{\mathbb{R}^n} |u(t)|^p \psi_R dx\\& =\int_{\mathbb{R}^n} \left(e^{-t (-\Delta)^m} u_0+\int_0^t e^{-(t-s) (-\Delta)^m}\left(|u(s)|^p\right) d s\right)\partial_t \psi_R dx\\
    &-\int_{\mathbb{R}^n} \left(e^{-t (-\Delta)^m} u_0+\int_0^t e^{-(t-s) (-\Delta)^m}\left(|u(s)|^p\right) d s\right)(-\Delta)^m\psi_Rdx.
\end{aligned}
    \end{equation*}
Since $u(t)$ is a mild solution \eqref{mild}, the right-hand side can be expressed more compactly as
\begin{equation*}
\begin{aligned}
    \frac{d}{dt}\int_{\mathbb{R}^n}u\psi_R dx-\int_{\mathbb{R}^n} |u(t)|^p\psi_Rdx =-\int_{\mathbb{R}^n} u(-\Delta)^m\psi_R dx+\int_{\mathbb{R}^n} u\partial_t \psi_R dx.
\end{aligned}
    \end{equation*}
Integrating this equality over the interval $[0,T]$ and noting that $\psi_R(\cdot,T)=0$, we obtain the weak identity
\begin{equation}\label{weak}
\begin{aligned}
    \varepsilon\int_{\mathbb{R}^n}u_0\psi_R(x,0) dx+\int_{Q_T} &|u|^p\psi_R dxdt=\int_{Q_T} u(-\Delta)^m\psi_R dxdt-\int_{Q_T} u\partial_t \psi_R dxdt,
\end{aligned}
    \end{equation}
where $Q_T:=[0,T] \times \mathbb{R}^n.$ From Lemma \ref{lifespanes}, we deduce
\begin{equation}\label{after}
  \begin{aligned}
\int_{Q_T}|u|^p \psi_R dx dt+ \varepsilon \int_{\mathbb{R}^n} u_0(x) \psi_R(0, x) dx  &\leq\int_{Q_T}|u|\left(\left|\partial_t \psi_R\right|+\left|(-\Delta)^m \psi_R\right|\right) dx dt \\
& \leq R^{-1}\int_{Q_T}|u|\left(\psi_R^*\right)^{\frac{1}{p}} dx dt. 
\end{aligned}  
\end{equation}
Applying Hölder’s inequality to the right-hand side, we obtain
\begin{equation*}
  \begin{aligned}
 R^{-1} \int_{Q_T}|u|\left(\psi_R^*\right)^{\frac{1}{p}} dx dt
& \leq R^{-1}\left(\int_{Q_T}|u|^p \psi_R^* dx dt\right)^{\frac{1}{p}}\left(\int_0^R \int_{|x|<R^{\frac{1}{2m}}} dx dt\right)^{\frac{1}{p^{\prime}}}\\
& \leq CR^{-\frac{p-1}{p}\left(\frac{1}{p-1}-\frac{n}{2m}\right)}\left(\int_{Q_T}|u|^p \psi_R^* dx dt\right)^{\frac{1}{p}}
\end{aligned}  
\end{equation*}
which holds for any $R \in\left(2 R_0, T\right)$. Consequently, we arrive at
\begin{equation}\label{without}
    \begin{aligned}
\int_{Q_T}|u|^p \psi_R dx dt+ &\varepsilon \int_{\mathbb{R}^n} u_0(x) \psi_R(0, x) dx \\&\leq  CR^{-\frac{p-1}{p}\left(\frac{1}{p-1}-\frac{n}{2m}\right)}\left(\int_{Q_T}|u|^p \psi_R^* dx dt\right)^{\frac{1}{p}}.
\end{aligned} 
\end{equation}
We now introduce the auxiliary quantities
$$
X(r):=\int_{Q_T}|u(t, x)|^p \psi_r(t, x) dx dt,$$ and $$Y(r):=\int_{Q_T}|u(t, x)|^p \psi_r^*(t, x) dx dt
$$
and define
$$
W(R):=\int_0^R Y(r) r^{-1} d r.
$$
By direct differentiation, it follows that $W'(R)=\tfrac{1}{R}Y(R)$. Using inequality \eqref{after}, we can express it in terms of these new quantities as
 \begin{equation}\label{new}
     X(R)+C\varepsilon  \lesssim R^{-\frac{p-1}{p}\left(\frac{1}{p-1}-\frac{n}{2m}\right)} Y(R)^{\frac{1}{p}},
 \end{equation}
where $C>0$ is constant comes from the fact that  $\int_{\mathbb{R}^n}u_0(x)>0$. Furthermore, from the properties of the cut-off functions (see Georgiev and Palmieri \cite[Lemma 1, p. 1011]{Palmieri}), we have the relation
\begin{equation}\label{Y(R)}
\frac{2}{\log 2} \int_0^R Y(r) \frac{\mathrm{d} r}{r} \leq X(R).
\end{equation}
Combining \eqref{new}–\eqref{Y(R)} yields
$$
\frac{2 W(R)}{\log 2}+C\varepsilon   \lesssim R^{-\frac{p-1}{p}\left(\frac{1}{p-1}-\frac{n}{2m}\right)+\frac{1}{p}}\left(W^{\prime}(R)\right)^{\frac{1}{p}}.
$$
This leads to the differential inequality
$$
 R^{\left(\frac{1}{p-1}-\frac{n}{2m}\right)(p-1)-1} \leq CW^{\prime}(R)\left(C\varepsilon +\frac{2 W(R)}{\log 2}\right)^{-p}.
$$
 Finally, note that the right-hand side can be rewritten in differential form as
\begin{equation}\label{qqq}
     R^{\left(\frac{1}{p-1}-\frac{n}{2m}\right)(p-1)-1} \leq \frac{C\log 2}{2(1-p)}\frac{d}{dR}\left(C\varepsilon +\frac{2 W(R)}{\log 2}\right)^{1-p}.
\end{equation}
To proceed, we integrate inequality \eqref{qqq} over the interval $\left[2 R_0, T\right]$. We obtain the following estimates for the left-hand side:
$$
\begin{aligned}
 \int_{2 R_0}^T R^{\left(\frac{1}{p-1}-\frac{n}{2m}\right)(p-1)-1} dR = \begin{cases}T^{\left(\frac{1}{p-1}-\frac{n}{2m}\right)(p-1)}-\left(2 R_0\right)^{\left(\frac{1}{p-1}-\frac{n}{2m}\right)(p-1)} & \text { if } 1<p<p_{\text {Fuj}}, \\{}\\
\log \left(\frac{T}{2 R_0}\right) & \text { if } p=p_{\mathrm{Fuj}}.\end{cases}
\end{aligned}
$$
For the right-hand side, integrating by parts yields
$$
\begin{aligned}
\frac{C\log 2}{2(1-p)} \int_{2 R_0}^T \frac{d}{dR}\left(C\varepsilon +\frac{2 W(R)}{\log 2}\right)^{1-p} dR \leq  C\varepsilon^{1-p} .
\end{aligned}
$$
After integrating \eqref{qqq} and taking $\varepsilon>0$ sufficiently small for the subcritical case $p<p_{\text {Fuj}}$, we obtain 
$$ 
T^{\left(\frac{1}{p-1}-\frac{n}{2m}\right)(p-1)} \leq C \varepsilon^{-(p-1)},$$ which implies $$T\lesssim \varepsilon^{-\left(\frac{1}{p-1}-\frac{n}{2m}\right)^{-1}}.
$$
When $p=p_{Fuj}$, we instead have

$$
\log \left(\frac{T}{2 R_0}\right) \leq \frac{\log 2}{2(p-1)}C^{1-p} \varepsilon^{-(p-1)},$$ hence $$T\lesssim \exp \left( \varepsilon^{-(p-1)}\right).
$$
Therefore, the statement of the theorem follows. 

\section*{Acknowledgments}
The authors are deeply grateful to A. Palmieri (Italy) and S. Tayachi (Tunisia) for their insightful recommendations on references concerning lifespan estimates for the parabolic equations. The authors also thank the anonymous reviewer for their thorough and thoughtful evaluation, as well as for the valuable comments that greatly enhanced the clarity and overall quality of this work.

\section*{Funding}
This research has been funded by the Science Committee of the Ministry of Education and Science of the Republic of Kazakhstan (Grant No. AP23483960).

\section*{Declaration of competing interest}  The authors declare that there is no conflict of interest.

\section*{Data Availability Statements} The manuscript has no associated data.

\bibliographystyle{ieeetr.bst}
\bibliography{references}
\end{document}